\documentclass[11pt,reqno]{amsart}

\usepackage{amssymb,amsthm,amsmath}

\numberwithin{equation}{section}

\theoremstyle{plain}
\newtheorem{prop}{Proposition}[section]
\newtheorem{lem}[prop]{Lemma}

\newtheorem{thm}[prop]{Theorem}

\theoremstyle{definition}

\newtheorem{question}[prop]{Question}

\theoremstyle{remark}
\newtheorem{remark}[prop]{Remark}

\newcommand\F{\mathbb{F}}
\newcommand\N{\mathbb{N}}
\newcommand\Q{\mathbb{Q}}
\newcommand\R{\mathbb{R}}
\newcommand\Z{\mathbb{Z}}
\newcommand\C{\mathbb{C}}

\begin{document}
\title[Diophantine approximation]{Diophantine approximation with restricted numerators and denominators
on semisimple groups}

\author{A.\@ Gorodnik}
\author{S.\@ Kadyrov}

\address[AG,SK]{School of Mathematics
University of Bristol
Bristol BS8 1TW, U.K.}
\email[AG]{a.gorodnik@bristol.ac.uk}
\email[SK]{shirali.kadyrov@bristol.ac.uk}

\begin{abstract}
We consider the problem of Diophantine approximation on semisimple algebraic groups 
by rational points with restricted numerators and denominators and establish a quantitative 
approximation result for all real points in the group by rational points
with a prescribed denominator and an almost prime numerator.
\vspace{0.2cm}

\vspace{0.5cm}

\begin{center}
\noindent {\sc Approximation diophantienne sur les groupes semi-simples avec num\'erateurs 
et d\'enominateurs restreintes}
\end{center}

\vspace{0.2cm}

\noindent Nous consid\'erons le probl\`eme de l'approximation diophantienne sur les groupes alg\'ebriques semi-simples
par des points rationnels avec les num\'e\-rateurs  et les d\'enominateurs restreintes et 
nous \'etablissons  un 
r\'esultat quantitatif d'approximation pour tous les points r\'eels dans le groupe par des points rationnels
avec un d\'enominateur sp\'ecifi\'e et un num\'erateur presque premier.

\end{abstract}

\maketitle

\section{Introduction}\label{sec:intro}

In this paper we are interested in the problem of Diophantine approximation
of real points $x$ by rational points $\frac{u}{v}$, where the numerator $u$ and the denominator $v$
are restricted to interesting arithmetic sets; for instance, when $u,v$ are primes or $r$-primes. 
Recall that an integer is called $r$-prime if it is a product of at most $r$ prime factors
counted with multiplicities.

Starting from the work of Vinogradov \cite{vi}, the question for which $\alpha>1$ the
inequality 
\begin{equation}\label{eq:D}
\left|x-\frac{u}{v}\right|\le v^{-\alpha},\quad x\in \R,
\end{equation}
has infinitely many solutions with integral $u$ and prime $v$ attracted significant attention
(see \cite{va77,h83,b86,j93,h96,hj02,m09}).
When $u$ is $r$-prime and $v$ is prime, this question has been investigated in
\cite{va76,h84}, but it still seems open when both $u$ and $v$ are 
assumed to be prime (see \cite{r,s}).

Here we consider an analogous question for the Diophantine approximation 
on semisimple algebraic groups. For instance, let us consider a special linear
group 
${\rm SL}_N=\{x\in \hbox{Mat}_N(\C):\, \det(x)=1\}.$
It is well known that ${\rm SL}_N(\Q)$ is dense in ${\rm SL}_N(\R)$,
and explicit quantitative density estimates have been established in \cite{ggn}.
Now it is natural to ask whether we can approximate any $x\in {\rm SL}_N(\R)$
by rational points $z\in {\rm SL}_N(\Q)$ whose coordinates have 
prescribed arithmetic properties. In particular, 

\begin{question}
{\it 
Is the set of points
in ${\rm SL}_N(\Q)$ with prime denominators and $r$-prime numerators dense in ${\rm SL}_N(\R)$?
}
\end{question}

As we shall show, this is indeed the case, and moreover, a quantitative estimate similar to (\ref{eq:D}) holds. 

In full generality, our result deals with a simply connected semisimple algebraic $\Q$-group
${\rm G}\subset \hbox{GL}_N$ which is split over $\Q$ and $\Q$-simple.
It is known that ${\rm G}(\Q)$ is dense in ${\rm G}(\R)$. Moreover, it
follows from the strong approximation property \cite[\S7.4]{pr} that for every
$n\ge 2$, the set ${\rm G}(\Z[1/n])$ is dense in ${\rm G}(\R)$.
Every $z\in \hbox{G}(\Q)$ can be uniquely written as $z=v^{-1}\, u$ with $v\in \N$
and $u\in \hbox{Mat}_N(\Z)$ such that $\gcd(u_{11},\ldots,u_{NN},v)=1$.
We use the denominator $\hbox{den}(z):=v$ to measure complexity of rational points
and quantify their density in $\hbox{G}(\R)$
with respect to a right invariant Riemannian metric $d$ on ${\rm G}(\R)$.

We note that it might happen that the set of $z\in \hbox{G}(\Q)$ such that $\hbox{den}(z)=n$ could be empty,
so that the problem of Diophantine approximation does not make sense for general denominators.
We say that an integer $n$ is \emph{admissible} if there exists $z\in \hbox{G}(\Q)$ such that $\hbox{den}(z)=n$.
Because of the weak approximation property the set of admissible denominators can be described in terms
of local abstractions. Namely, $n=\prod_p p^{\alpha_p}$ is admissible if for every $p$,
there exists $g\in {\rm G}(\Q_p)$ such that $\|g\|_p=p^{\alpha_p}$, where $\|\cdot\|_p$ denotes the
maximal norm.

Let $f_1,\ldots,f_t$ be a collection of polynomials on $\hbox{Mat}_N(\C)$ with integral
coefficients. We assume that $f_i$'s considered as elements of the coordinate ring 
$\Q[{\rm G}]$ are nonzero, distinct, and absolutely irreducible.
We say that an element $z\in {\rm Mat}_N(\Z[1/n])$ is {\it $r$-prime},
with respect to the family of polynomials $f_1,\ldots,f_t$,
if $f_1(z)\cdots f_t(z)$ is a product of at most $r$ prime factors in the ring $\Z[1/n]$.

Our main result is the following

\begin{thm}\label{thm:main_short}
Given a simply connected semisimple algebraic
group $\rm G\subset {\rm GL}_N$ defined over $\Q$, which is split over $\Q$
and $\Q$-simple, and
a collection of polynomials $f_1,\ldots,f_t$ as above,
there exist explicit $\alpha>0$ and $r\in \N$ such that
for every $x\in {\rm G}(\R)$ and admissible $n\ge n_0(x)$, one can find
$z\in {\rm G}(\Q)$ satisfying 
\begin{align*}
&d(x,z)\le n^{-\alpha},\\
&\hbox{\rm den}(z)=n,
\end{align*}
and $r$-prime in ${\rm Mat}_n(\Z[1/n])$.
Moreover, the constant $n_0(x)$ is uniform over $x$ in bounded subsets of ${\rm G}(\R)$.
\end{thm}

\begin{remark}
Let $\|\cdot\|_\infty$ be the Euclidean norm on $\hbox{Mat}_N(\R)$.
Using that ${\rm G}(\R)$ is a submanifold of $\hbox{Mat}_N(\R)$, one can check that 
$$
\|x_1-x_2\|_\infty \ll d(x_1,x_2),\quad x_1,x_2\in{\rm G}(\R),
$$
where the implied constant is uniform over $x_1,x_2$ in bounded subset of 
${\rm G}(\R)$. Hence, Theorem \ref{thm:main_short} also implies a Diophantine approximation
result with respect to the Euclidean norm.
\end{remark}


Next we consider the case when $\rm G$ is not necessarily split over $\Q$.
Then the set
${\rm G}(\Z[1/p])$, where $p$ is prime, might be discrete in ${\rm G}(\R)$. 
In fact,  ${\rm G}(\Z[1/p])$ is dense in ${\rm G}(\R)$
if and only if the group ${\rm G}$ is 
isotropic over the $p$-adic field $\Q_p$ (see \cite[Th.~7.12]{pr}). Under this assumption we  
prove a weaker version of Theorem \ref{thm:main_short} for
${\rm G}(\Z[1/p])\subset {\rm G}(\R)$ where the parameters 
$\alpha,r,n_0(x)$ might depend on $p$.

\begin{thm}\label{thm:main_short2}
Given a simply connected $\Q$-simple algebraic
group $\rm G\subset {\rm GL}_N$ defined over $\Q$, a collection of polynomials $f_1,\ldots,f_t$  as above,
and a finite collection $\mathcal{P}$ of primes such that $\rm G$ is
isotropic over $\Q_p$ for all $p\in \mathcal{P}$,
there exist explicit $\alpha>0$ and $r\in \N$ such that 
for every $x\in {\rm G}(\R)$ and admissible $n\ge n_0(x)$ whose prime divisors are in $\mathcal{P}$, one can find
$z\in {\rm G}(\Q)$ satisfying 
\begin{align*}
&d(x,z)\le n^{-\alpha},\\
&\hbox{\rm den}(z)=n,
\end{align*}
and $r$-prime in ${\rm Mat}(\Z[1/n])$.
Moreover, the constant $n_0(x)$ is uniform over $x$ in bounded subsets of ${\rm G}(\R)$.
\end{thm}

More explicit statements of Theorems \ref{thm:main_short} 
and \ref{thm:main_short2} are given in Section \ref{sec:end} below.

The proof of the main theorems is based on the uniform spectral gap property for the 
automorphic unitary representations
and the asymptotic analysis of suitable averaging operators
combined with standard number-theoretic sieving arguments.
In the following section we introduce essential notation and outline the strategy of the proof
in more details.

\subsection*{Acknowledgements}
The first author is support by EPSRC, ERC and RCUK, and the second author is
supported by EPSRC. We would like to thank A.~Haynes for useful comments.

\section{Initial set-up}\label{sec:notation}

Throughout the paper, $p$ always denotes a prime number. 

Let ${\rm G}\subset {\rm GL}_N$ be a simply connected $\Q$-simple algebraic group defined over $\Q$.
We also use the same notation for the corresponding integral model of $\rm G$ defined by the embedding
of $\rm G$ into ${\rm GL}_N$.

For $n\in\N$, we set
\begin{align*}
G_n^{\rm f}:=\prod_{p|n} {\rm G}(\Q_p)\quad\hbox{and}\quad
G_n:= {\rm G}(\R)\times G_n^{\rm f}.
\end{align*}
Let
$$
\Gamma_n:={\rm G}(\Z [1/n]).
$$
We consider $\Gamma_n$ as a subgroup of $G_n$ embedded in $G_n$ diagonally.
Then $\Gamma_n$ is a discrete subgroup with finite covolume (see \cite[\S5.4]{pr}).

For every prime $p$, we fix a special maximal compact open subgroup $U_p$ of ${\rm G}(\Q_p)$
(as defined in \cite{bt0,tits}), so that $U_p={\rm G}(\Z_p)$ for almost all $p$. 

We denote by $m_{{\rm G}(\Q_p)}$ the 
Haar measure on ${\rm G}(\Q_p)$ normalised so that 
$$
m_{{\rm G}(\Q_p)}(U_p)=1.
$$
The Haar measure $m_{G_n^{\rm f}}$ on $G_n^{\rm f}$ is the product of the measures
$m_{{\rm G}(\Q_p)}$ over $p$ dividing $n$. The Haar measure $m_{G_n}$ on $G_n$
is the product of a Haar measure $m_{{\rm G}(\R)}$ on ${\rm G}(\R)$ and the measure
$m_{G_n^{\rm f}}$. 

For $q\in \N$, coprime to $n$, 
we define the congruence subgroups
$$
\Gamma_n(q):=\{\gamma \in \Gamma_n : \gamma =\hbox{id}\bmod q \}.
$$ 
Clearly, $\Gamma_n(q)$ is a finite index normal subgroup of $\Gamma_n$, and the space 
$$
Y_{n,q}:=G_n/\Gamma_n(q)
$$
has finite volume.
For simplicity, we also set $Y_{n}:=G_n/\Gamma_n$.
We denote by $m_{Y_{n,q}}$ the invariant measure on $Y_{n,q}$ induced by $m_{G_n}$
and by $\mu_{Y_{n,q}}$ the probability invariant measure on $Y_{n,q}$, so that
$\mu_{Y_{n,q}}=\frac{m_{Y_{n,q}}}{m_{Y_{n,q}}(Y_{n,q})}$.

We denote by $\|\cdot\|_p$ the maximum norm on $\hbox{Mat}_N(\Q_p)$.
Given $n\in \N$ with prime decomposition $n=\prod_p p^{\alpha_p}$, we set 
\begin{equation}\label{eq:bnp}
B_{n,p} :=\{g\in {\rm G}(\Q_p):\, \|g\|_p=p^{\alpha_p}\}\quad\hbox{and}\quad
B_n^{\rm f}:=\prod_{p|n} B_{n,p}.
\end{equation}
Throughout the paper we always assume that the integer $n$ is admissible, namely,
the set $B_n^{\rm f}$ is not empty.
Note that $B_{n,p}$ is a compact subset of ${\rm G}(\Q_p)$, which is invariant under
the compact open subgroup ${\rm G}(\Z_p)$. In particular, $B_{n,p}$ is invariant under $U_p$
for almost all $p$.
We fix a right-invariant Riemannian metric on ${\rm G}(\R)$. For $x\in {\rm G}(\R)$ and $\epsilon>0$, we set
\begin{align*}
B^\infty (x,\epsilon)&:=\{g \in {\rm G}(\R):\, d(g,x)\le \epsilon\},\\
B_n(x,\epsilon) &:= B^\infty (x,\epsilon)\times B_n^{\rm f}.
\end{align*}

We denote by $L^2(Y_{n,q})=L^2(Y_{n,q},\mu_{Y_{n,q}})$ the Hilbert space of square-integrable functions
on $Y_{n,q}$, and by $L_0^2(Y_{n,q})$ the subspace of functions with zero integral.
For $p=\infty$ and a prime $p$ dividing $n$, 
the group ${\rm G}(\Q_p)$ naturally acts on $Y_{n,q}$, and 
we denote by $\pi_{Y_{n,q},p}$ the corresponding unitary representation of ${\rm G}(\Q_p)$
on $L^2(Y_{n,q})$. We denote by $\pi_{Y_{n,q}}$ and 
$\pi_{Y_{n,q},{\rm f}}$ the unitary representations of $G_n$ and $G_n^{\rm f}$ on $L^2(Y_{n,q})$
respectively. It would be also convenient to denote by 
$\pi_{Y_{n,q},p}^0$, $\pi^0_{Y_{n,q}}$, $\pi^0_{Y_{n,q},{\rm f}}$
the restrictions of the above representations to $L_0^2(Y_{n,q})$.

Given a unitary representation $\pi$ of a locally compact group $G$ on a Hilbert space
$\mathcal{H}$ and a finite Borel measure $\beta$ on $G$, we denote by
$\pi(\beta):\mathcal{H}\to\mathcal{H}$ the corresponding averaging operator 
defined by
$$
\pi(\beta)v=\int_G \pi(g)v\, d\beta(g), \quad v\in \mathcal{H}.
$$
We note that if $\beta$ is a probability measure, then $\|\pi(\beta)\|\le 1$.

The crucial ingredient of our argument is  the study of suitable averaging operators on the space $L^2(Y_{n,q})$.
Namely, we denote by $\beta_{n,x,\epsilon}$ the uniform probability measure on $G_n$ supported on the set
$B_n(x,\epsilon)$. This defines an averaging operator
\begin{align}
\pi_{Y_{n,q}}(\beta_{n,x,\epsilon}):L^2(Y_{n,q}) &\to L^2(Y_{n,q}): \label{eq:beta}\\
 f\;\quad&\mapsto
\frac{1}{m_{G_n}(B_n(x,\epsilon))}\int_{B_n(x,\epsilon)} \pi_{Y_{n,q}}(g)f\, dm_{G_n}(g).\nonumber
\end{align}
We note that this operator is well defined only for admissible $p$.

A unitary representation $\pi$ of a locally compact group $G$ on a Hilbert space
$\mathcal{H}$ is called {\it $L^{r+}$ integrable} if there exists a dense family of vectors
$v_1,v_2\in \mathcal{H}$ such that the function
$$
g\mapsto \left<\pi(g)v_1,v_2\right>,\quad g\in G,
$$
is in $L^{r+\delta}(G)$ for every $\delta>0$.
In our setting, it follows from the property $(\tau)$, established in \cite{cl}, that
the representations $\pi_{Y_{n,q},p}$, restricted to $L_0^2(Y_{n,q})$, are uniformly integrable.
Namely, there exists $r\ge 2$, independent of $n,q,p$, such that
all the representations $\pi_{Y_{n,q},p}$, restricted to $L_0^2(Y_{n,q})$, are $L^{r+}$ integrable.
We denote by $r({\rm G})$ the least real number with this property.
Let $\iota({\rm G})$ be the least even integer greater than or equal to $r({\rm G})/2$ if $r({\rm G})>2$,
and $\iota({\rm G})=1$ if $r({\rm G})=2$.

\subsection*{Outline of the proof}
In Section \ref{sec:ergodic}, we analyse the asymptotic behaviour of the averaging
operators $\pi_{Y_{n,q}}(\beta_{n,x,\epsilon})$ and establish a quantitative mean ergodic theorem
for them, namely, an  estimate of the form 
\begin{equation}\label{eq:0}
\left\|\pi_{Y_{n,q}}(\beta_{n,x,\epsilon})|_{L_0^2(Y_{n,q})}\right\|\ll
m_{G^{\rm f}_n}(B_n^{\rm f})^{-\theta},
\end{equation}
where $\theta>0$ is determined by the integrability exponent $r({\rm G})$.
This argument is based on the techniques developed in \cite{GorNev1, GorNev2},
but it is crucial for our application that the implied constant in (\ref{eq:0})
is uniform on $n$, and this requires additional considerations.

Section \ref{sec:vol} plays auxiliary role. In this section we establish several volume 
estimates which might be of independent interest. We use these estimates in 
the later sections to guarantee uniformity in the parameter $n$.

In Section \ref{sec:counting}, we use (\ref{eq:0}) to estimate 
the cardinality of elements in $\bar\gamma \Gamma_n(q)$ lying in the regions
$B_n(x,\epsilon)$, for admissible $n$. Typically, such a counting estimate
requires that the regions are well-rounded in the sense of \cite[Def.~1.1]{GorNev2},
but the regions $B_n(x,\epsilon)$ are not well-rounded as $\epsilon\to 0^+$.
Nonetheless, we shall establish a quantitative estimate for 
$|B_n(x,\epsilon)\cap \bar\gamma \Gamma_n(q)|$ as $n\to \infty$.

Finally, in Section \ref{sec:end}, we use a combinatorial sieving argument as in
\cite{HalRic,NevSar,GorNev} to estimate the cardinality of almost prime points lying in the regions
$B_n(x,\epsilon)$. This leads to the proof of the main theorems.


\section{Averaging operators}
\label{sec:ergodic}
 
In this section, we study the asymptotic behaviour of
the averaging operators $\pi_{Y_{n,q}}(\beta_{n,x,\epsilon})$ defined in (\ref{eq:beta})
for admissible $n$.
Similar problem has been previously investigated in \cite{GorNev1,GorNev2}.
In particular, the following theorem can be proved by
adopting the methods of \cite[Th.~4.5]{GorNev2}.

\begin{thm}\label{prop:QMET0} For every $\eta>0$ and $f\in L^2(Y_{n,q})$,
\begin{align*}
\left\|\pi_{Y_{n,q}}(\beta_{n,x,\epsilon})f-\int_{Y_{n,q}}f\, d\mu_{Y_{n,q}}\right\|_{2}
\ll_{n,\eta} \,
m_{G^{\rm f}_n}(B_n^{\rm f})^{-(2\iota({\rm G}))^{-1}+\eta}\|f\|_{2},
 \end{align*}
where the implied constant depends only on the set of prime divisors of $n$.
\end{thm}

\begin{proof}
The statement of the theorem is equivalent to the estimate
\begin{align*}
\left\|\pi^0_{Y_{n,q}}(\beta_{n,x,\epsilon})\right\|
\ll_{n,\eta} \,
m_{G^{\rm f}_n}(B_n^{\rm f})^{-(2\iota({\rm G}))^{-1}+\eta},\quad \eta>0.
 \end{align*}
We observe that the probability measure $\beta_{n,x,\epsilon}$ decomposes as a product
$$
\beta_{n,x,\epsilon}=\beta^\infty_{x,\epsilon}\otimes\left( \bigotimes_{p|n}\beta_{n,p}\right),
$$
where 
$\beta_{x,\epsilon}^\infty$ is the uniform probability measure on ${\rm G}(\R)$ supported on $B^\infty(x,\epsilon)$,
and $\beta_{n,p}$'s are the uniform probability measures on ${\rm G}(\Q_p)$ supported on $B_{n,p}$. 
This implies that $\pi^0_{Y_{n,q}}(\beta_{n,x,\epsilon})$ can be written as a product of commuting operators
$$
\pi^0_{Y_{n,q}}(\beta_{n,x,\epsilon})
=\pi^0_{Y_{n,q},\infty}(\beta^\infty_{x,\epsilon})\prod_{p|n}\pi^0_{Y_{n,q},p}(\beta_{n,p}).
$$
Since $\|\pi^0_{Y_{n,q},\infty}(\beta^\infty_{x,\epsilon})\|\le 1$, we obtain
\begin{equation}\label{eq:prod1}
\left\|\pi^0_{Y_{n,q}}(\beta_{n,x,\epsilon})\right\|
\le\prod_{p|n}\left\|\pi^0_{Y_{n,q},p}(\beta_{n,p})\right\|.
\end{equation}

To estimate the norms of $\pi^0_{Y_{n,q},p}(\beta_{n,p})$, we use the argument as 
in \cite[Th.~4.5]{GorNev2}
(which is based on Nevo's spectral transfer principle \cite[Th.~1]{Nev}). 
Explicitly, for real-valued functions $f_1,f_2$ in $L^2_0(Y_{n,q})$,
it follows from the Jensen's inequality that
\begin{align*}
\left< \pi^0_{Y_{n,q},p}(\beta_{n,p})f_1,f_2\right>^{\iota({\rm G})}
&=\left(\int_{{\rm G}(\mathbb{Q}_p)} \left< \pi^0_{Y_{n,q},p}(g)f_1,f_2\right>\, d\beta_{n,p}(g)\right)^{\iota({\rm G})}\\
&\le\int_{{\rm G}(\mathbb{Q}_p)} \left< \pi^0_{Y_{n,q},p}(g)f_1,f_2\right>^{\iota({\rm G})}\, d\beta_{n,p}(g)\\
&=\int_{{\rm G}(\mathbb{Q}_p)} \left< (\pi^0_{Y_{n,q},p})^{\otimes\iota({\rm G})}(g)f_1^{\otimes\iota({\rm G})},f_2^{\otimes\iota({\rm G})}\right>\, d\beta_{n,p}(g)\\
&=\left< (\pi^0_{Y_{n,q},p})^{\otimes\iota({\rm G})}(\beta_{n,p})f_1^{\otimes\iota({\rm G})},f_2^{\otimes\iota({\rm G})}\right>\\
&\le \left\| (\pi^0_{Y_{n,q},p})^{\otimes\iota({\rm G})}(\beta_{n,p})\right\| \|f_1\|^{\otimes\iota({\rm G})}\|f_2\|^{\otimes\iota({\rm G})}.
\end{align*}
Hence,
\begin{equation}\label{eq:nevo}
\left\|\pi^0_{Y_{n,q},p}(\beta_{n,p})\right\|\le \left\| (\pi^0_{Y_{n,q},p})^{\otimes\iota({\rm G})}(\beta_{n,p})\right\|^{1/\iota({\rm G})}.
\end{equation}
Since $\iota({\rm G})\ge r({\rm G})/2$, the representation 
$(\pi^0_{Y_{n,q},p})^{\otimes\iota({\rm G})}$ is $L^{2+}$ integrable, and by \cite[Th.~1]{chh},
\begin{equation}\label{eq:nevo1}
\left\| (\pi^0_{Y_{n,q},p})^{\otimes\iota({\rm G})}(\beta_{n,p})\right\|\le 
\left\| \lambda_{{\rm G}(\mathbb{Q}_p)}(\beta_{n,p}) \right\|,
\end{equation}
where $\lambda_{{\rm G}(\mathbb{Q}_p)}$ denotes the regular representation of 
${\rm G}(\mathbb{Q}_p)$ on $L^2({\rm G}(\mathbb{Q}_p))$.
By the Kunze--Stein inequality, for every $f\in L^2({\rm G}(\mathbb{Q}_p))$ and $q\in [1,2)$,
$$
\|\beta_{n,p}*f\|_2\ll_{p,q} \|\beta_{n,p}\|_q\|f\|_2=m_{{\rm G}(\mathbb{Q}_p)}(B_{n,q})^{1/q-1}\|f\|_2.
$$
For semisimple simply connected groups over non-Archimedean fields, this inequality was proven in
\cite{v}. Then it follows that
$$
\left\| \lambda_{{\rm G}(\mathbb{Q}_p)}(\beta_{n,p}) \right\|\ll_{p,\eta} m_{{\rm
    G}(\mathbb{Q}_p)}(B_{n,q})^{-1/2+\eta},\quad \eta>0,
$$
and
\begin{equation}\label{eq:prod2_0}
\left\|\pi^0_{Y_{n,q},p}(\beta_{n,p})\right\|\ll_{p,\eta}\, m_{{\rm G}(\Q_p)}(B_{n,p})^{-(2\iota({\rm G}))^{-1}+\eta},\quad
\eta>0,
\end{equation}
which completes the proof of the theorem.
\end{proof}

We note that the crucial estimate (\ref{eq:prod2_0}) can be deduced from
the spherical Kunze--Stein inequality 
(see, for instance, \cite[Prop.~5.9]{GorNev1}), and the implied constant in (\ref{eq:prod2_0})
can be estimated explicitly, but unfortunately it blows up as $p\to \infty$.

For our purposes we need the following uniform version of Theorem \ref{prop:QMET0}.

\begin{thm}\label{prop:QMET} For every $\eta>0$ and $f\in L^2(Y_{n,q})$,
\begin{align*}
\left\|\pi_{Y_{n,q}}(\beta_{n,x,\epsilon})f-\int_{Y_{n,q}}f\, d\mu_{Y_{n,q}}\right\|_{2}\ll_\eta\,
m_{G^{\rm f}_n}(B_n^{\rm f})^{-(4\iota({\rm G}))^{-1}+\eta}\|f\|_{2}.
 \end{align*}
\end{thm}

\begin{proof}
As in the above proof, we need to estimate 
$\|\pi^0_{Y_{n,q}}(\beta_{n,x,\epsilon})\|$, and 
because of (\ref{eq:prod1}), it is sufficient to give a bound on the norms of 
$\pi^0_{Y_{n,q},p}(\beta_{n,p})$. We claim that
\begin{equation}\label{eq:prod2}
\left\|\pi^0_{Y_{n,q},p}(\beta_{n,p})\right\|\le c_{p,\eta}\, m_{{\rm G}(\Q_p)}(B_{n,p})^{-(4\iota({\rm G}))^{-1}+\eta},\quad
\eta>0,
\end{equation}
where the constant $c_{p,\eta}\ge 1$ satisfies 
$$
\prod_p c_{p,\eta}<\infty.
$$
Clearly, (\ref{eq:prod1}) combined with (\ref{eq:prod2}) implies the theorem.
 
To prove (\ref{eq:prod2}), we use that by the estimates \eqref{eq:nevo}--\eqref{eq:nevo1}, 
\begin{equation}\label{eq:est}
\left\|\pi^0_{Y_{n,q},p}(\beta_{n,p})\right\|\le 
\left\|\lambda_{{\rm G}(\Q_p)}(\beta_{n,p})\right\|^{1/\iota({\rm G})}.
\end{equation}

Let $\tilde B_{n,p} :=U_pB_{n,p} U_p$ and $\tilde \beta_{n,p}$ be the uniform probability measure
supported on $\tilde B_{n,p}$. 
Recall that for almost all $p$, we have $U_p={\rm G}(\Z_p)$.
For those $p$, we have $\tilde B_{n,p}=B_{n,p}$ and
\begin{equation}\label{eq:l0}
\left\|\lambda_{{\rm G}(\Q_p)}(\beta_{n,p})\right\|=
\left\|\lambda_{{\rm G}(\Q_p)}(\tilde\beta_{n,p})\right\|.
\end{equation}
To deal with the remaining finite set of primes, we observe that
$B_{n,p}\subset \tilde B_{n,p}$, and hence for every $f\in L^2({\rm G}(\Q_p))$,
\begin{align*}
\left\|\lambda_{{\rm G}(\Q_p)}(\beta_{n,p})f\right\|&\le 
\left\|\lambda_{{\rm G}(\Q_p)}(\beta_{n,p})|f|\right\| \\
&\le \frac{m_{{\rm G}(\Q_p)}(\tilde B_{n,p})}{m_{{\rm G}(\Q_p)}(B_{n,p})}\cdot
\left\|\lambda_{{\rm G}(\Q_p)}(\tilde \beta_{n,p})|f|\right\|.
\end{align*}
Hence,
\begin{equation}\label{eq:l}
\left\|\lambda_{{\rm G}(\Q_p)}(\beta_{n,p})\right\|\le 
\frac{m_{{\rm G}(\Q_p)}(\tilde B_{n,p})}{m_{{\rm G}(\Q_p)}(B_{n,p})}\cdot
\left\|\lambda_{{\rm G}(\Q_p)}(\tilde \beta_{n,p})\right\|.
\end{equation}
Since $U_p$ and ${\rm G}(\Z_p)$ are both compact open subgroups in ${\rm G}(\Q_p)$,
they are commensurable, and it follows that for some $c_p\ge 1$,
$$
m_{{\rm G}(\Q_p)}(\tilde B_{n,p})\le c_p\, m_{{\rm G}(\Q_p)}(B_{n,p}).
$$
Hence, it follows from (\ref{eq:l0}) and (\ref{eq:l}) that
\begin{align}\label{eq:est3}
\left\|\lambda_{{\rm G}(\Q_p)}(\beta_{n,p})\right\|\le c_p
\left\|\lambda_{{\rm G}(\Q_p)}(\tilde \beta_{n,p})\right\|,
\end{align}
where $c_p=1$ for almost all $p$.

Since $\tilde \beta_{n,p}$ is bi-invariant under $U_p$,
we can estimate the norm 
$\left\|\lambda_{{\rm G}(\Q_p)}(\tilde\beta_{n,p})\right\|$ using Herz' majoration argument
(see \cite[Prop.~5.9]{GorNev1}). 
Indeed, since $U_p$'s are special subgroups, it follows from the structure theory \cite[3.3.2]{tits}
that there exists a closed amenable subgroup $Q_p$ such that 
\begin{equation}\label{eq:iwasawa}
{\rm G}(\Q_p)=U_pQ_p,
\end{equation}
i.e.,
${\rm G}(\Q_p)$ is an Iwasawa group in the sense of \cite[Def.~5.1(1)]{GorNev1}. 
By \cite[Prop.~5.9]{GorNev1}, 

\begin{align*}\label{eq:csp}
\|\lambda_{{\rm G}(\Q_p)}(\tilde\beta_{n,p})\|
=\int_{{\rm G}(\mathbb{Q}_p)} \Xi_p\, d\beta_{n,p},
\end{align*}
where $\Xi_p$
is the Harish-Chandra function
$\Xi_p$ (see \cite[Def.~5.1(2)]{GorNev1}).
We recall that $\Xi_p$ belongs to $L^{2+\epsilon}({\rm G}(\mathbb{Q}_p))$ for every $\epsilon>0$,
so that by H\"older inequality, for every $s\in [1,2)$ and $t=(1-1/s)^{-1}$,
\begin{equation}\label{eq:csp}
\|\lambda_{{\rm G}(\Q_p)}(\tilde\beta_{n,p})\|
\le \|\Xi_p\|_{t}\cdot\|\tilde\beta_{n,p}\|_s  
= \|\Xi_p\|_{t} \cdot m_{{\rm G}(\Q_p)}(\tilde B_{n,p})^{-1+1/s}.
\end{equation}
We set $a_{p,s}=\|\Xi_p\|_{t}$.

To estimate $a_{p,s}$, we recall the explicit formula for the Harish-Chandra function.
Let $\Delta_{p}$ denote the modular function on the group $Q_p$.
For $g\in {\rm G}(\Q_p)$, we denote by $q(g)$ the $Q_p$-component of $g$ with 
respect to the decomposition (\ref{eq:iwasawa}).
The Harish-Chandra function on ${\rm G}(\Q_p)$ is defined by
$$
\Xi_p(g)=\int_{U_p}\Delta_{p}(q(gu))^{-1/2}\, dm_{{\rm G}(\Q_p)}(u).
$$
We have
$$
a_{p,s}\ge \Xi_p(e) m_{{\rm G}(\Q_p)}(U_p)^{1/t}=1,
$$
and by \cite[Prop.~6.3]{GorNev2}, when $t>4$ (that is, when $s<4/3$),
$$
\prod_p a_{p,s}<\infty.
$$ 
Combining (\ref{eq:est}), (\ref{eq:est3}), and (\ref{eq:csp}), we conclude that 
$$
\left\|\pi^0_{Y_{n,q},p}(\beta_{n,p})\right\|\le (a_{p,s}c_p)^{1/\iota({\rm G})}\,m_{{\rm G}(\Q_p)}(
B_{n,p})^{(-1+1/s)/\iota({\rm G})},
$$
where 
$$
\prod_p (a_{p,s}c_p)^{1/\iota({\rm G})}<\infty
$$
for $s<4/3$.
This implies (\ref{eq:prod2}) and completes the proof of the theorem.
\end{proof}

\section{Volume estimates}\label{sec:vol}

This section plays auxiliary role and can be skipped for the first reading.
Here we prove uniform estimates for the volumes of the sets $Y_{n}$ and the sets $B_n^{\rm f}$.

\begin{prop} \label{l:vol}
$$
\inf_{n\in \N} m_{Y_{n}}(Y_{n})>0\quad\hbox{and}\quad
\sup_{n\in \N} m_{Y_{n}}(Y_{n})<\infty.
$$
\end{prop}

\begin{proof}
For $n\in \N$ we define
\begin{equation}\label{eq:o}
\mathcal O_n^{\rm f}:=\prod_{p | n} {\rm G}(\Z_p),
\end{equation}
which is a compact open subgroup of $G_n^{\rm f}$.

To prove the first claim, we fix a sufficiently small open subset $\mathcal{O}^\infty$ of ${\rm G}(\R)$
and set $\mathcal{O}_n:=\mathcal{O}^\infty\times \mathcal{O}_n^{\rm f}$.
We claim that $\mathcal{O}_n$ injects into $Y_{n}=G_n/\Gamma_n$ under the projection map
$G_n\to G_n/\Gamma_n$.
Indeed, if for some $\gamma\in\Gamma_n$, we have $\mathcal{O}_n\gamma\cap \mathcal{O}_n\ne\emptyset$,
then it follows that $\gamma\in {\rm G}(\Z)$. Therefore, it is sufficient to take $\mathcal{O}^\infty$ which
injects into ${\rm G}(\R)/{\rm G}(\Z)$. Then 
$$
m_{Y_{n}}(Y_{n})\ge m_{{\rm G}(\R)}(\mathcal{O}^\infty)\prod_{p|n} m_{{\rm G}(\Q_p)}({\rm G}(\Z_p)).  
$$
Since $m_{{\rm G}(\Q_p)}({\rm G}(\Z_p))=m_{{\rm G}(\Q_p)}(U_p)=1$ for almost all $p$, this proves the first claim.

Now we turn to the proof of the second claim.
We fix a prime $p_0$ such that ${\rm G}$ is isotropic over $\Q_{p_0}$ and $U_{p_0}={\rm G}(\Z_{p_0})$
(such a prime exists by \cite[Th.~6.7]{pr}).
We consider the two separate cases depending on whether $p_0$ divides $n$ or not.

Suppose that $p_0$ divides $n$ and write $n=p_0^\alpha n_0$ with $n_0$ coprime to $p_0$.
We identify $G_{p_0}$ and $G_{n_0}^{\rm f}$ with the corresponding subgroups
of $G_n$, so that 
\begin{equation}\label{eq:decomp}
G_n=G_{p_0}G_{n_0}^{\rm f}.
\end{equation}
Since ${\rm G}$ is isotropic over $\Q_{p_0}$, it follows that $G_{p_0}$ is not compact,
and by the strong approximation property \cite[\S7.4]{pr}, $G_{p_0}\Gamma_n$ is dense in $G_n$.
In particular, it follows that for every $g\in G_n$,
$$
G_{p_0}\Gamma_n\cap O_{n_0}^{\rm f}g\ne \emptyset.
$$
This proves that
$$
G_n=G_{p_0}O^{\rm f}_{n_0}\Gamma_n.
$$
Let $\Omega_{p_0}$ be a measurable fundamental set in $G_{p_0}$ for the right action of $\Gamma_{p_0}$.
Then $G_{p_0}=\Omega_{p_0} \Gamma_{p_0}$, and every element $g\in G_{n}$ can be written 
with respect to the decomposition (\ref{eq:decomp}) as
$$
g=(\omega\delta,o)\cdot (\gamma,\gamma)=(\omega,o\delta^{-1})\cdot (\delta\gamma,\delta\gamma),
$$
where $\omega\in \Omega_{p_0}$, $\delta\in \Gamma_{p_0}$, $o\in O_{n_0}^{\rm f}$, and $\gamma\in \Gamma_n$.
Since $o\delta^{-1}\in O_{n_0}^{\rm f}$ and $\delta\gamma\in \Gamma_n$, this shows that
$$
G_n=\Omega_{p_0} O^{\rm f}_{n_0}\Gamma_n.
$$
Therefore, 
\begin{equation}\label{eq:ll}
m_{Y_{n}}(Y_{n})\le m_{G_n}(\Omega_{p_0} O^{\rm f}_{n_0})=
m_{G_{p_0}}(\Omega_{p_0})\cdot m_{G_{n_0}^{\rm f}}(O_{n_0}^{\rm f}).
\end{equation}
We observe that
$$
m_{G_{n_0}^{\rm f}}(O^{\rm f}_{n_0})=\prod_{p|n_0} m_{{\rm G}(\Q_p)}({\rm G}(\Z_p)).
$$
Since ${\rm G}(\Z_p)$ is compact, $m_{{\rm G}(\Q_p)}({\rm G}(\Z_p))<\infty$.
Moreover, for almost $p$, we have ${\rm G}(\Z_p)=U_p$ and $m_{{\rm G}(\Q_p)}({\rm G}(\Z_p))=1$.
Hence, the second claim of the lemma follows from (\ref{eq:ll}).

Now suppose that $p_0$ does not divide $n$. In this case, we 
identify $G_n$ and ${\rm G}(\Q_{p_0})$ with the corresponding subgroups
of $G_{np_0}$, so that 
\begin{equation}\label{eq:dec2}
G_{np_0}=G_n {\rm G}(\Q_{p_0}).
\end{equation}
Let $\Omega_n$ be a measurable fundamental set in $G_n$ for the right action of $\Gamma_n$.
We claim that the natural projection map 
\begin{equation}\label{eq:map}
\Omega_n\times {\rm G}(\Z_{p_0})\to Y_{np_0}=G_{np_0}/\Gamma_{np_0}
\end{equation}
defined by the decomposition (\ref{eq:dec2}) is one-to-one.
Indeed, suppose that for some $g_1,g_2\in \Omega_n$ and $u_1,u_2\in {\rm G}(\Z_{p_0})$,
there exists $\gamma\in \Gamma_{np_0}$ such that $g_1\gamma=g_2$ and $u_1\gamma=u_2$.
Then 
$$
\gamma=u_1^{-1}u_2\in \Gamma_{n p_0}\cap {\rm G}(\Z_{p_0})=\Gamma_n
$$
and since $\Omega_n$ is a fundamental set, it follows that $\gamma=e$.
This proves the claim.
It is clear the map (\ref{eq:map}) sends the product measure $m_{G_n}\otimes m_{{\rm G}(\Q_{p_0})}$
to the measure $m_{Y_{np_0}}$. Hence, we obtain
$$
m_{G_n}(\Omega_n)m_{{\rm G}(\Q_{p_0})}({\rm
  G}(\Z_{p_0}))\le m_{Y_{np_0}}(Y_{np_0}).
$$
Since it follows from the previous paragraph that the right hand side is uniformly bounded,
we conclude that $m_{Y_n}(Y_n)=m_{G_n}(\Omega_n)$ is uniformly bounded as well.
This completes the proof of the proposition.
\end{proof}

We say that a number $n$ is {\it isotropic} if for every prime divisor $p$ of $n$,
the group $\rm G$ is isotropic over $\Q_p$. In particular, if $\rm G$ is isotropic over $\Q$,
then every number is isotropic.

\begin{prop}\label{p:b_n}
There exists $a>0$ such that for every isotropic admissible number $n$,
$$
m_{G_n^{\rm f}}(B_n^{\rm f})\ge c(n)\, n^{a},
$$
where $c(n)>0$ depends only on the set of prime divisors of $n$.

Moreover, if ${\rm G}$ is split over $\Q$, then $c(n)$ is independent of $n$.
\end{prop}

\begin{proof}
Let $n=\prod_p p^{\alpha_p}$ be the prime decomposition of $n$.
Since $B_n^{\rm f}$ is the product of the sets $B_{n,p}$ (see (\ref{eq:bnp})), it is sufficient to prove that
there exists $c_p>0$ such that
$$
m_{{\rm G}(\Q_p)}(B_{n,p})\ge c_p\, (p^{\alpha_p})^a,
$$
When ${\rm G}$ is split over $\Q$, we show that $c_p=1$ for almost all $p$.

Recall that $U_p$ is a special maximal compact subgroup.
Therefore, by \cite[3.3.3]{tits}, we have a Cartan decomposition
$$
{\rm G}(\Q_p)=U_p{\rm Z}_p(\Q_p)U_p,
$$
where ${\rm Z}_p$ is the centraliser of a maximal $\Q_p$-split torus ${\rm S}_p$ in ${\rm G}$.
We fix a set $\Pi_p$ of (restricted) simple roots for ${\rm G}$ with respect to ${\rm S}_p$, and set
$$
S_p^+=\{s\in {\rm S}(\Q_p):\, |\chi(s)|_p\ge 1\quad \hbox{for $\chi\in \Pi_p$}\}.
$$
Then the Cartan decomposition takes form
\begin{equation}\label{eq:cartan}
{\rm G}(\Q_p)=U_pS_p^+\Omega_pU_p,
\end{equation}
where $\Omega_p$ is a finite subset of ${\rm Z}_p(\Q_p)$.
Since $U_p$ is compact and $\Omega_p$ is finite, there exists $c_p'>0$ such that 
for $g=u_1s\omega u_2\in U_pS_p^+\Omega_pU_p$, we have 
\begin{equation}\label{eq:l1}
\|g\|\le c_p'\|s\|_p.
\end{equation}
Consider the representation $\rho:{\rm G}\to \hbox{GL}_N$ 
corresponding to the embedding ${\rm G}\subset \hbox{GL}_N$ and denote by $\Phi_p$
the set of weights of this representation with respect to ${\rm S}_p$. 
Since ${\rm S}_p$ is split over $\Q_p$, the action of ${\rm S}_p(\Q_p)$ on $\Q_p^N$ 
is completely reducible, and 
\begin{equation}\label{eq:l2}
\|s\|_p\le c_p''\cdot \max_{\xi\in \Phi_p} |\xi(s)|_p
\end{equation}
for some $c_p''>0$.

Let $\Pi_p^\vee$ denote the set of fundamental weights corresponding to $\Pi_p$.
A weight $\xi$ is called dominant if 
$$
\xi=\prod_{\psi\in \Pi_p^\vee} \psi^{n_\psi}
$$
for some nonnegative integers $n_{\psi}$. It follows from
\cite[Ch.~3, \S1.9]{ov} that there exists $k_0\in \N$ such that
every weight $\xi$ can be written as
\begin{equation}\label{eq:xi0}
\xi^{k_0}=\prod_{\chi\in \Pi_p} \chi^{m_\chi}
\end{equation}
with $m_\chi\in \Z$. Moreover, if $\xi$ is dominant, then
the integers $m_\chi$ are non-negative.

By \cite[Th.~7.2]{bt1}, there exists a semisimple subgroup $\tilde {\rm G}_p$ of $\rm G$
which is split over $\Q_p$, contains ${\rm S}_p$ as a maximal torus, and the set
$\Pi_p$ forms the set of simple roots for ${\rm S}_p$ in $\tilde{\rm G}_p$.
The linear representations of $\tilde{\rm G}_p$ 
defined over $\Q_p$ are described by the theory of highest weights.
In particular, it follows from the description of possible weights
(see \cite[Ch.~3, \S2.2]{ov})
that the maximum in (\ref{eq:l2}) can be taken over the dominant weights
in $\Phi_p$. Moreover, it follows from the classification of 
semisimple groups and their representations that 
the set of all possible weights appearing in $\rho|_{{\rm S}_p}$ for some $p$ is finite.
Let $\Delta_p$ be the product of all positive roots of ${\rm S}_p$ in ${\rm G}$
counted with multiplicities. Then we deduce from (\ref{eq:xi0}) that there exists $\ell \in \N$,
independent of $p$, such that
\begin{equation}\label{eq:l3}
\xi\le \Delta_p^\ell
\end{equation}
for all dominant weights $\xi$ appearing in $\rho|_{{\rm S}_p}$.

Now combining (\ref{eq:l1}), (\ref{eq:l2}) and (\ref{eq:l3}), we deduce that for all $g=u_1s\omega u_2\in
U_pS_p^+\Omega_pU_p$, we have
$$
\|g\|_p\le  (c_p'c_p'')\,|\Delta_p(s)|^\ell,
$$
and when $g\in B_{n,p}$, we obtain
\begin{equation}\label{eq:d1}
|\Delta_p(s)|\ge  (c_p'c_p'')^{-1/\ell}\,(p^{\alpha_p})^{1/\ell}.
\end{equation}
By \cite[3.2.15]{m71},
\begin{equation}\label{eq:d2}
m_{{\rm G}(\Q_p)}(U_psU_p)\ge |\Delta_p(s)|.
\end{equation}
Since both ${\rm G}(\Z_p)$ and $U_p$ are compact open subgroups,
${\rm G}(\Z_p)\cap U_p$ has finite index in $U_p$, and there exists an open
normal subgroup $V_p$ of $U_p$ contained in ${\rm G}(\Z_p)$. Then
for $g=u_1s\omega u_2\in U_pS_p^+\Omega_pU_p$, we obtain
\begin{align}\label{eq:l5}
m_{{\rm G}(\Q_p)}(V_pgV_p)&=m_{{\rm G}(\Q_p)}(u_1V_p s(\omega V_p\omega^{-1})\omega u_2)\\
&=m_{{\rm G}(\Q_p)}(V_p s(\omega V_p\omega^{-1})).\nonumber
\end{align}
By compactness, both $V_p$ and $V_p\cap \omega V_p\omega^{-1}$ have finite index in $U_p$.
Therefore, it follows that
\begin{equation}\label{eq:l6}
m_{{\rm G}(\Q_p)}(U_psU_p)\le c_p''' \,m_{{\rm G}(\Q_p)}(V_ps(\omega V_p\omega^{-1}))
\end{equation}
for some $c_p'''>0$. Hence, when $g\in B_{n,p}$, we deduce from (\ref{eq:d1})--(\ref{eq:d2}) 
combined with (\ref{eq:l5})--(\ref{eq:l6}) that
$$
m_{{\rm G}(\Q_p)}(V_pgV_p)\ge (c_p'c_p'')^{-1/\ell} (c_p''')^{-1}\,(p^{\alpha_p})^{1/\ell}.
$$
Since $n$ is admissible, there exists $g\in B_{n,p}$, and 
since $V_p\subset {\rm G}(\Z_p)$, we have $V_pgV_p\subset B_{n,p}$. 
This implies the first claim of the proposition.
 
To prove the second claim,
we assume that ${\rm G}$ is split over $\Q$. Then the Cartan decomposition (\ref{eq:cartan})
is of the form ${\rm G}(\Q_p)=U_pS_p^+U_p$.
For almost all $p$, we can take ${\rm S}_p$ to be a maximal $\Q$-split torus $\rm S$.
Then the action of ${\rm S}(\Q)$ on $\Q^N$ is completely reducible.
Since for almost all $p$, we have $U_p={\rm G}(\Z_p)$,
the estimate (\ref{eq:l1}) holds with $c_p'=1$ for almost all $p$.
Since the action of ${\rm S}$ is completely reducible over $\Q$, for $s\in {\rm S}$,
$$
\rho(s)=\sum_{\xi\in \Phi_p} \xi(s) v_\xi
$$
for some $v_\xi \in \hbox{Mat}_N(\Q)$. Hence, in the estimate (\ref{eq:l2}), 
$$
c_p''=\max_{\xi} \|v_\xi\|_p,
$$
and it is clear that for almost all $p$, $c_p''=1$. Finally, in the argument (\ref{eq:l5})--(\ref{eq:l6}),
we can take $V_p=U_p={\rm G}(\Z_p)$ for almost all $p$. Therefore, for almost all $p$, we obtain
$$
m_{{\rm G}(\Q_p)}(B_{n,p})\ge (p^{\alpha_p})^{1/\ell},
$$
which completes the proof of the proposition.
\end{proof}

\section{Counting estimates}\label{sec:counting}

In this section we prove an estimate on the number of lattice points in the regions $B_n(x,\epsilon)$
with admissible $n$.

\begin{thm}\label{thm:count} 
For every coprime $n,q\in\N$ with admissible $n$, $\bar\gamma\in \Gamma_n$, $x\in {\rm G}(\R)$, $\kappa,\eta>0$, and
$\epsilon\in (0,\epsilon_0(\kappa,x)]$, 
the following estimate holds
\begin{align*}
|B_n(x,\epsilon)\cap \bar\gamma\Gamma_n(q)|
=&\frac{m_{G_n}(B_n(x,\epsilon))}{m_{Y_{n,q}}(Y_{n,q})}+O_x\left( \epsilon^{\kappa+\dim({\rm G})} m_{G_n^{\rm f}}(B_n^{\rm f}) \right)\\
&+O_{x,\eta}\left(\epsilon^{-\kappa\dim({\rm G})}m_{G_n^{\rm f}}(B_n^{\rm f})^{1-(4\iota({\rm G}))^{-1}+\eta}\right),
\end{align*}
where $\epsilon_0(\kappa,x)$ and the implied constants are uniform over $x$ in bounded sets.
\end{thm}

This result will be deduced from Theorem \ref{prop:QMET} following the strategy of \cite{GorNev2}.

We set
\begin{equation}\label{eqn:Oepsilon}
\mathcal O_n(\epsilon):=B^\infty(e,\epsilon) \times \prod_{p|n} {\rm G}(\Z_p),
\end{equation}
where $B^\infty(e,\epsilon)$ is the $\epsilon$-neighbourhood of identity in ${\rm G}(\R)$
with respect to the Riemannian metric $d$ on ${\rm G}(\R)$. Since the metric $d$ is right invariant,
$\mathcal O_n(\epsilon)$ is a symmetric neighbourhood of identity in $G_n$.

\begin{lem}\label{l:o}
For every $n\in\N$ and $\epsilon\in (0,1]$,
$$
\epsilon^{\dim({\rm G})}\ll m_{G_n} (\mathcal O_n(\epsilon)) \ll \epsilon^{\dim({\rm G})}.
$$
\end{lem}

\begin{proof}
Writing $B^\infty(e,\epsilon)$ in the exponential coordinates for the Riemannian metric, we deduce that
\begin{equation}\label{eq:mm}
\epsilon^{\dim({\rm G})}\ll m_{{\rm G}(\R)} (B^\infty(e,\epsilon)) \ll \epsilon^{\dim({\rm G})}.
\end{equation}
We recall that the measures $m_{{\rm G}(\Q_p)}$ are normalised so that
 $m_{{\rm G}(\Q_p)}(U_p)=1$, and $U_p={\rm G}(\Z_p)$ for all but finitely many primes $p$. 
For the remaining primes, we observe that since $U_p$ and ${\rm G}(\Z_p)$ are compact open subgroups in ${\rm G}(\Q_p)$,
it follows that $U_p\cap {\rm G}(\Z_p)$ has finite index in both $U_p$ and ${\rm G}(\Z_p)$.
Therefore, 
$$
1\ll \prod_{p|n} m_{{\rm G}(\Q_p)}({\rm G}(\Z_p)) \ll 1.
$$
Combining this estimate with (\ref{eq:mm}), we deduce the lemma.
\end{proof}

\begin{lem}\label{lem:well}
We have 
\begin{enumerate}
\item[(i)] For every $n\in \N$, $x\in {\rm G}(\R)$, and $\epsilon,\epsilon'\in (0,1]$,
\begin{equation*}
\mathcal O_n(\epsilon') B_n(x,\epsilon) \mathcal O_n(\epsilon') \subset B_n(x,\epsilon+c_1(x)\epsilon'),
\end{equation*}
where $c_1(x)$ is uniform over $x$ in bounded sets.

\item[(ii)]
For every $n\in \N$, $x\in {\rm G}(\R)$, and $\epsilon,\epsilon'\in (0,\epsilon_0(x)]$,
$$
m_{G_n}(B_n(x,\epsilon+\epsilon') )\le m_{G_n}(B_n(x,\epsilon))+c_2(x) \epsilon'\epsilon^{\dim({\rm
    G})-1}\, m_{G^{\rm f}_n}(B_n^{\rm f}),
$$
where $\epsilon_0(x)$ and $c_2(x)$ are uniform over $x$ in bounded sets.
\end{enumerate}
\end{lem}

\begin{proof}
To prove (i), we observe that 
$B_n(x,\epsilon)=B^\infty(x,\epsilon) \times \prod_{p|n} B_{n,p}$ and the sets $B_{n,p}$ are invariant
under ${\rm G}(\Z_p).$ Thus, it suffices to prove that
for every $u_1,u_2 \in B^\infty(e,\epsilon')$ and $b \in B^\infty(x,\epsilon)$, we have
\begin{equation}\label{eq:i}
d(x,u_1bu_2) \le \epsilon+c_1(x) \epsilon'.
\end{equation}
Using the right invariance of the Riemannian metric on ${\rm G}(\R)$, we obtain
\begin{align*}
d(x,u_1 b u_2)&=d(x (b u_2)^{-1}, u_1)\le d(x (b u_2)^{-1},e)+d(e,u_1)\\
&\le d(x,b u_2)+\epsilon' \le d(x b^{-1}, b u_2 b^{-1})+\epsilon'\\
&\le d(x b^{-1},e)+d(e, b u_2 b^{-1})+\epsilon'\le \epsilon+d(e, b u_2 b^{-1})+\epsilon'.
\end{align*}
Since $d(e, b u_2 b^{-1})\ll_b d(e,u_2)\le \epsilon'$ where the implied constant is uniform over $b$ in bounded sets,
we deduce (\ref{eq:i}).

To prove (ii), it is sufficient to show that
$$
m_{{\rm G}(\R)}(B^\infty(x,\epsilon+\epsilon'))- m_{{\rm G}(\R)}(B^\infty(x,\epsilon)) \le c_2(x) \epsilon'
\epsilon^{\dim({\rm G})-1}.
$$
This follows from the disintegration formula for the measure $m_{{\rm G}(\R)}$
as in \cite[p.~66]{sak}.
\end{proof}

Let $\chi_{n,\epsilon}$ be the constant multiple of the characteristic function of
the set $\mathcal{O}_n(\epsilon)$ which is normalised so that
$\int_{G_n} \chi_{n,\epsilon} \,\,dm_{G_n}=1$. For $\bar\gamma \in \Gamma_n$, we also define a function $\phi_{n,q,\epsilon}^{\bar\gamma}$
on $Y_{n,q}=G_n/\Gamma_n(q)$ by
$$
\phi_{n,q,\epsilon}^{\bar\gamma}(g\Gamma_n(q)):=\sum_{\gamma \in
  \Gamma_n(q)}\chi_{n,\epsilon}(g\gamma\bar\gamma )
=\sum_{\delta \in
  \bar \gamma \Gamma_n(q)}\chi_{n,\epsilon}(g\delta ).
$$
Note that $\bar\gamma\Gamma_n(q)=\Gamma_n(q)\bar\gamma$ because 
$\Gamma_n(q)$ is normal in $\Gamma_n$.
Clearly, $\phi_{n,q,\epsilon}^{\bar\gamma}$ is a bounded measurable function on $Y_{n,q}$
with compact support.

The following lemma shows that averages of this function can be used to approximate the cardinality
$|B_n(x,\epsilon) \cap \bar\gamma\Gamma_n(q)|$.

\begin{lem}\label{lem:boundpm}
For every coprime $n,q\in\N$, $\bar\gamma\in \Gamma_n$, $x\in {\rm G}(\R)$, 
$\epsilon,\epsilon'\in (0,1]$, and $h\in \mathcal O_n(\epsilon')$,
\begin{align*}
|B_n(x,\epsilon) \cap \bar\gamma\Gamma_n(q)|&\le 
\int_{B_n(x,\epsilon+c_1(x)\epsilon')}\phi_{n,q,\epsilon'}^{\bar\gamma }(g^{-1} h \Gamma_n(q))\,
dm_{G_n}(g),\\
|B_n(x,\epsilon) \cap \bar\gamma\Gamma_n(q)|&\ge 
\int_{B_n(x,\epsilon - c_1(x)\epsilon')} \phi_{n,q,\epsilon'}^{\bar\gamma }(g^{-1} h \Gamma_n(q))\,
dm_{G_n}(g).
 \end{align*}
\end{lem}


\begin{proof}
Suppose that for some $\delta\in B_n(x,\epsilon)\cap \bar \gamma \Gamma_n(q)$ and $h\in \mathcal O_n(\epsilon')$,
we have $\chi_{n,\epsilon'}(g^{-1}h\delta )\ne 0$. Then by Lemma \ref{lem:well}(i),
\begin{equation}\label{eqg}
g\in h\delta \mathcal O_n(\epsilon')^{-1}\subset O_n(\epsilon')B_n(\epsilon)O_n(\epsilon')\subset B_n(x,\epsilon+c_1(x)\epsilon').
\end{equation}
Since the function $\chi_{n,\epsilon'}$ is non-negative, and 
\begin{equation}\label{i}
\int_{G_n} \chi_{n,\epsilon'}
(g^{-1})\,dm_{G_n}(g)=
\int_{G_n} \chi_{n,\epsilon'} \,\,dm_{G_n}=1,
\end{equation}
it follows from (\ref{eqg}) that 
\begin{align*}
&\int_{B_n(x,\epsilon+c_1(x)\epsilon')}\phi_{n,q,\epsilon'}^{\bar\gamma }(g^{-1} h \Gamma_n(q))\,
dm_{G_n}(g)\\
\ge & 
\sum_{\delta \in B_n(x,\epsilon) \cap
  \bar \gamma \Gamma_n(q)} \int_{B_n(x,\epsilon+c_1(x)\epsilon')} \chi_{n,\epsilon'}(g^{-1}h \delta )\,
dm_{G_n}(g)\\
= & 
\sum_{\delta \in B_n(x,\epsilon) \cap
  \bar \gamma \Gamma_n(q)} \int_{G_n} \chi_{n,\epsilon'}(g^{-1}h \delta )\,
dm_{G_n}(g)
\ge |B_n(x,\epsilon) \cap \bar\gamma\Gamma_n(q)|.
\end{align*}
This proves the first inequality.

To prove the second inequality, we observe that if 
$\chi_{n,\epsilon'}(g^{-1}h \delta )\ne 0$ for some $g\in B_n(x,\epsilon - c_1(x)\epsilon')$,
$h\in \mathcal O_n(\epsilon')$ and $\delta\in \bar \gamma \Gamma_n(q)$, then
by Lemma \ref{lem:well}(i),
$$
\delta\in h^{-1}g \mathcal O_n(\epsilon')\subset O_n(\epsilon')B_n(\epsilon - c_1(x)\epsilon')O_n(\epsilon') \subset B_n(x,\epsilon).
$$
This implies that
\begin{align*}
&\int_{B_n(x,\epsilon - c_1(x)\epsilon')} \phi_{n,q,\epsilon'}^{\bar\gamma }(g^{-1} h \Gamma_n(q))\,
dm_{G_n}(g)\\
=&\sum_{\delta \in B_n(x,\epsilon) \cap
  \bar \gamma \Gamma_n(q)} \int_{B_n(x,\epsilon-c_1(x)\epsilon')} \chi_{n,\epsilon'}(g^{-1}h \delta )\,
dm_{G_n}(g)\\
\le &
\sum_{\delta \in B_n(x,\epsilon) \cap
  \bar \gamma \Gamma_n(q)} \int_{G_n} \chi_{n,\epsilon'}(g^{-1}h \delta )\,
dm_{G_n}(g)=|B_n(x,\epsilon) \cap
  \bar \gamma \Gamma_n(q)|,
\end{align*}
where we used that $\chi_{n,\epsilon'}\ge 0$ and (\ref{i}). Hence, the second inequality also holds.
\end{proof}

\begin{proof}[Proof of Theorem~\ref{thm:count}]
We first show that there exists $\theta_0>0$ such that for any distinct $\gamma_1,\gamma_2 \in \Gamma_n$, 
\begin{equation}\label{eqn:Odisj}
\mathcal O_n(\theta_0)\gamma_1 \cap \mathcal O_n(\theta_0)\gamma_2=\emptyset.
\end{equation}
Indeed, suppose that for some 
$$
(g_\infty,g_f),(h_\infty,h_f) \in \mathcal O_n(\theta)=
B^\infty(e,\theta) \times \prod_{p|n} {\rm G}(\Z_p)\quad\hbox{and}\quad \gamma \in \Gamma_n,
$$
we have
$$
(g_\infty,g_f)=(h_\infty,h_f)(\gamma,\gamma).
$$
Then 
$$
(\gamma,\gamma)=(g_\infty h_\infty^{-1},g_f h_f^{-1})\in B^\infty(e,\theta)^2 \times \prod_{p|n} {\rm G}(\Z_p).
$$
In particular, we conclude that $\gamma\in {\rm G}(\Z)$, and hence $\gamma=e$ if $\theta$ is
sufficiently small.

It follows from  \eqref{eqn:Odisj} for every $\theta\in (0,\theta_0]$,
\begin{equation}\label{eqn:Oeps}
\mu_{Y_{n,q}}(\mathcal O_n(\theta) \Gamma_n(q))=\frac{m_{G_n}(\mathcal O_n(\theta))}{m_{Y_{n,q}}(Y_{n,q})}.
\end{equation}
Moreover,
$$
\int_{Y_{n,q}}\phi_{n,q,\theta}^{\bar\gamma }\,d\mu_{Y_{n,q}}
=\int_{G_n} \chi_{n,\theta}(g)\,\frac{dm_{G_n}(g)}{m_{Y_{n,q}}(Y_{n,q})}=\frac{1}{m_{Y_{n,q}}(Y_{n,q})},
$$
and similarly,
$$
\|\phi_{n,q,\theta}^{\bar\gamma }\|_2^2
=\int_{G_n} \chi_{n,\theta}^2(g)\,\frac{dm_{G_n}(g)}{m_{Y_{n,q}}(Y_{n,q})}=\frac{m_{G_n}(\mathcal O_n(\theta))^{-1}}{m_{Y_{n,q}}(Y_{n,q})}.
$$
By Theorem~\ref{prop:QMET}, for every $\rho,\eta>0$, there exists $c_\eta>0$ such that
\begin{align*}
&\left\|\pi_{Y_{n,q}}(\beta_{n,x,\rho})\phi_{n,q,\theta}^{\bar\gamma}-\int_{Y_{n,q}}\phi_{n,q,\theta}^{\bar\gamma}\,
  d\mu_{Y_{n,q}}\right\|_{2}\\
\le\,&\, c_\eta\, m_{G^{\rm f}_n}(B_n^{\rm f})^{-(4\iota({\rm G}))^{-1}+\eta}\|\phi_{n,q,\theta}^{\bar\gamma}\|_{2}.
 \end{align*}
Therefore, we deduce that for every $\delta>0$,
\begin{align*}
&\mu_{Y_{n,q}}\left(\left\{h\Gamma_n(q) : \left|\pi_{Y_{n,q}}(\beta_{n,x,\rho})\phi_{n,q,\theta}^{\bar\gamma }(h\Gamma_n(q))-\frac{1}{  m_{Y_{n,q}}(Y_{n,q})  }\right|>\delta\right\}\right) \\
\le\, &\, c_\eta^2 \delta^{-2}\frac{m_{G_n}(\mathcal
  O_n(\theta))^{-1}}{m_{Y_{n,q}}(Y_{n,q})}m_{G^{\rm f}_n}(B_n^{\rm f})^{-(2\iota({\rm G}))^{-1}+2\eta}.
\end{align*}
Let us take $\delta>0$ such that 
\begin{equation}\label{eqn:ht}
\mu_{Y_{n,q}}(\mathcal O_n(\theta) \Gamma_n(q)) > c_\eta^2 \delta^{-2}\frac{m_{G_n}(\mathcal
  O_n(\theta))^{-1}}{m_{Y_{n,q}}(Y_{n,q})}m_{G^{\rm f}_n}(B_n^{\rm f})^{-(2\iota({\rm G}))^{-1}+2\eta}.
\end{equation}
Note that it follows from \eqref{eqn:Oeps} that we may choose $\delta$ so that
\begin{align*}
\delta &=O_\eta \left(m_{G_n}(\mathcal O_n(\theta))^{-1}m_{G^{\rm f}_n}(B_n^{\rm
    f})^{-(4\iota({\rm G}))^{-1}+\eta}\right)\\
&= O_\eta \left(\theta^{-\dim({\rm G})}m_{G^{\rm f}_n}(B_n^{\rm
    f})^{-(4\iota({\rm G}))^{-1}+\eta}\right),
\end{align*}
where we used Lemma \ref{l:o}.
Then we deduce from (\ref{eqn:ht}) that
there exists $h \in \mathcal O_n(\theta)$ satisfying
$$
\left|\pi_{Y_{n,q}}(\beta_{n,x,\rho})\phi_{n,q,\theta}^{\bar\gamma
  }(h\Gamma_n(q))-\frac{1}{  m_{Y_{n,q}}(Y_{n,q})  }\right|\le \delta,
$$
which gives
$$
\left|\int_{B_n(x,\rho)}\phi_{n,q,\theta}^{\bar\gamma
  }(g^{-1}h\Gamma_n(q))\, dm_{G_n}(g) -\frac{m_{G_n}(B_n(x,\rho))}{  m_{Y_{n,q}}(Y_{n,q})  }\right|\le
\delta\, m_{G_n}(B_n(x,\rho)).
$$
Therefore,
\begin{align}\label{eq:fff}
&\int_{B_n(x,\rho)}\phi_{n,q,\theta}^{\bar\gamma
  }(g^{-1}h\Gamma_n(q))\, dm_{G_n}(g)\\
 = &
 \frac{m_{G_n}(B_n(x,\rho))}{  m_{Y_{n,q}}(Y_{n,q})  }
+O_\eta \left(\theta^{-\dim({\rm G})}m_{G_n}(B_n(x,\rho)) m_{G^{\rm f}_n}(B_n^{\rm
    f})^{-(4\iota({\rm G}))^{-1}+\eta}\right).\nonumber
\end{align}

Now to finish the proof of the theorem, we observe that according to Lemma~\ref{lem:boundpm},
$$
|B_n(x,\epsilon) \cap \bar\gamma\Gamma_n(q) | \le
\int_{B_n(x,\epsilon+c_1(x)\epsilon^{\kappa+1})}\phi_{n,q,\epsilon^{\kappa+1}}^{\bar\gamma }(g^{-1} h
\Gamma_n(q))\, dm_{G_n}(g).
$$
Combining this estimate with \eqref{eq:fff}, we obtain
\begin{align}\label{eq:f_f}
&|B_n(x,\epsilon) \cap \bar\gamma\Gamma_n(q) | 
\le \frac{m_{G_n}(B_n(x,\epsilon+c_1(x)\epsilon^{\kappa+1}))}{  m_{Y_{n,q}}(Y_{n,q})  }\\
&+O_\eta \left(\epsilon^{-(\kappa+1) \dim({\rm G})}m_{G_n}(B_n(x,\epsilon+c_1(x)\epsilon^{\kappa+1})) m_{G^{\rm f}_n}(B_n^{\rm
    f})^{-(4\iota({\rm G}))^{-1}+\eta}\right).\nonumber
\end{align}
By Lemma \ref{lem:well}(ii), for sufficiently small $\epsilon>0$,
\begin{align*}
m_{G_n}(B_n(x,\epsilon+c_1(x)\epsilon^{\kappa+1}))
=m_{G_n}(B_n(x,\epsilon)) +O_x\left(\epsilon^{\kappa+\dim({\rm G})} m_{G^{\rm f}_n}(B_n^{\rm f})\right),
\end{align*}
where the implied constants are uniform over $x$ in bounded sets.
Also, we have 
\begin{align*}
m_{G_n}(B_n(x,\epsilon+c_1(x)\epsilon^{\kappa+1}))
=O_x\left(\epsilon^{\dim({\rm G})} m_{G^{\rm f}_n}(B_n^{\rm f})\right).
\end{align*}
Therefore, (\ref{eq:f_f}) implies that
\begin{align*}
|B_n(x,\epsilon) \cap \bar\gamma\Gamma_n(q) | 
\le &\frac{m_{G_n}(B_n(x,\epsilon))}{  m_{Y_{n,q}}(Y_{n,q})  }+
O_x\left(\epsilon^{\kappa+\dim({\rm G})} m_{G^{\rm f}_n}(B_n^{\rm f})\right)\\
&+O_{x,\eta} \left(\epsilon^{-\kappa \dim({\rm G})} m_{G^{\rm f}_n}(B_n^{\rm
    f})^{1-(4\iota({\rm G}))^{-1}+\eta}\right).
\end{align*}
Here we used that $m_{Y_{n,q}}(Y_{n,q})\gg 1$ which follows from Proposition \ref{l:vol}.
This proves the required upper bound for $|B_n(x,\epsilon) \cap \bar\gamma\Gamma_n(q) |$.

The proof of the lower bound is similar,
and we use the second estimate from Lemma \ref{lem:boundpm}.
Note that in this proof we need to arrange that
$\epsilon-c_1(x)\epsilon^{\kappa+1}>0$ which holds for sufficiently small $\epsilon$,
depending on $\kappa,x$.
\end{proof}

\section{Completion of the proof}\label{sec:end}

In this section, we finish the proof of our main results stated in the Introduction.
It would be convenient to introduce a parameter
\begin{equation}\label{eq:A}
a({\rm G}):=\limsup_{n\to \infty} {}'\;\, \frac{\log m_{G_n^{\rm f}}(B_n^{\rm f})}{\log n},
\end{equation}
where the $\limsup$ is taken over all admissible integers $n$.
Note that according to Theorem \ref{thm:count}, the quantity $a({\rm G})$ measures
the polynomial growth rate of the number of rational points in $\rm G$ with given denominator 
and lying in a bounded subset of ${\rm G}(\R)$. By Proposition \ref{p:b_n}, $a({\rm G})>0$
if $\rm G$ is split over $\Q$. Given a finite set $\mathcal{P}$ of prime numbers, we also set
$$
a({\rm G},\mathcal{P}):=\limsup_{n\to \infty}{}' \;\,\frac{\log m_{G_n^{\rm f}}(B_n^{\rm f})}{\log n},
$$
where the $\limsup$ is taken over all admissible integers $n$ with prime divisors in $\mathcal{P}$.
If the group ${\rm  G}$ is isotropic over $\Q_p$ for every $p\in \mathcal{P}$, then
$a({\rm G},\mathcal{P})>0$ by Proposition \ref{p:b_n}.

Throughout this section we use the following simplified notation:
$$
d:=\dim({\rm G}),\;\;
a:=a({\rm G}),\;\; a_{\mathcal{P}}:=a({\rm G},\mathcal{P}),\;\;
\iota:=\iota({\rm G}).
$$
Recall that $\iota({\rm G})$ is computed in terms of the integrability exponent 
of automorphic representations (see Section \ref{sec:notation}).

For an integral polynomial $f$, we denote by $\Delta_n(f)$ the positive integer, coprime to $n$,
representing the greatest common divisor of $f(\gamma)$, $\gamma\in \Gamma_n$, in the ring $\Z[1/n]$.
We denote by $\delta_n(f)$ the number of prime factors of $\Delta_n(f)$.
Note that $\Delta_n(f)$ divides $\Delta_1(f)$ and $\delta_n(f)\le \delta_1(f)$.

The following result is a more precise version of Theorem \ref{thm:main_short}.

\begin{thm}\label{thm:mainn}
Let $\rm G\subset {\rm GL}_N$ be a simply connected $\Q$-simple algebraic
group  defined over $\Q$, which is split over $\Q$ and
$f_1,\ldots,f_t$ a collection of polynomials as in the Introduction.
Then for every $x\in {\rm G}(\R)$, $\alpha\in (0,\alpha_0)$ with 
$\alpha_0:=d^{-1}a(4\iota)^{-1}$ and admissible
$n\ge n_0(\alpha,x)$, there exists $z\in {\rm G}(\Q)$ satisfying
\begin{align*}
&d(x,z)\le n^{-\alpha},\\
&\hbox{\rm den}(z)=n,
\end{align*}
and $r$-prime in ${\rm Mat}_N(\Z[1/n])$,
where
$$
r=\delta_n(f_1\cdots f_t)+\left\lceil \frac{9 t\deg(f_1\cdots f_t)(d+1)^2}{a(4\iota)^{-1}-\alpha d} \right\rceil.
$$
The constant $n_0(\alpha,x)$ is uniform over $x$ in bounded  subsets of ${\rm G}(\R)$.
\end{thm}

\begin{proof}
By Theorem \ref{thm:count},
for every coprime $n,q\in\N$ with admissible $n$, $\bar\gamma\in \Gamma_n$, $x\in {\rm G}(\R)$, $\kappa,\eta>0$, and
$\epsilon\in (0,\epsilon_0(\kappa,x)]$, 
\begin{align}\label{eq:count}
|B_n(x,\epsilon)\cap \bar\gamma\Gamma_n(q)|
=&\frac{m_{G_n}(B_n(x,\epsilon))}{m_{Y_{n,q}}(Y_{n,q})}+O_x\left( \epsilon^{\kappa+d} m_{G_n^{\rm f}}(B_n^{\rm f}) \right)\\
&+O_{x,\eta}\left(\epsilon^{-\kappa d}m_{G_n^{\rm f}}(B_n^{\rm f})^{1-(4\iota({\rm G}))^{-1}+\eta}\right),\nonumber
\end{align}
where the implied constants are uniform over $x$ in bounded sets.
We apply this estimate with $\epsilon_n=m_{G_n^{\rm f}}(B_n^{\rm f})^{-\alpha'}$ 
where $n$ is sufficiently large, $\alpha'\in (0,\alpha_0')$, and
$\alpha_0':=\alpha_0/a=d^{-1}(4\iota)^{-1}$.
To optimise the error term in (\ref{eq:count}), we choose
\begin{equation}\label{eq:kappa}
\kappa=\frac{(4\iota)^{-1}-\alpha'd}{\alpha'(d+1)}.
\end{equation}
Note that the parameter $\alpha_0'$ is chosen to guarantee that $\kappa>0$.
Then (\ref{eq:count}) becomes
\begin{equation}\label{eq:main2}
|B_n(x,\epsilon_n)\cap \bar\gamma\Gamma_n(q)|
=\frac{m_{G_n}(B_n(x,\epsilon_n))}{m_{Y_{n,q}}(Y_{n,q})}
+O_{x,\eta}\left( m_{G_n^{\rm f}}(B_n^{\rm f})^{1-\alpha'(\kappa+d)+\eta} \right).
\end{equation}
It would be convenient to  set 
$$
T_n(x):=|B_n(x,\epsilon_n)\cap \Gamma_n|.
$$
Since
$$
m_{{\rm G}(\R)}(B^\infty(x,\epsilon))=m_{{\rm G}(\R)}(B^\infty(e,\epsilon))\gg \epsilon^d,
$$
we have 
\begin{equation}\label{eq:bbn}
m_{G_n}(B_n(x,\epsilon_n))\gg m_{G^{\rm f}_n}(B^{\rm f}_n)^{1-\alpha' d}.
\end{equation}
Using (\ref{eq:bbn}) and Proposition \ref{l:vol},
we deduce from (\ref{eq:main2}) with sufficiently small $\eta>0$ that 
\begin{align}\label{eq:TT}
T_n(x) &\gg 
\frac{m_{G^{\rm f}_n}(B^{\rm f}_n)^{1-\alpha' d}}{m_{Y_{n}}(Y_{n})}
+O_{x,\eta}\left(m_{G_n^{\rm f}}(B_n^{\rm f})^{1-\alpha'(\kappa+d)+\eta}\right)\\
&\gg  m_{G^{\rm f}_n}(B^{\rm f}_n)^{1-\alpha' d}\nonumber
\end{align}
when $n$ is sufficiently large.
In particular, it follows from the definition of $a=a({\rm G})$ that for every $b\in (0,a)$ and sufficiently large $n$,
\begin{equation}\label{eq:final}
T_n(x) \gg n^{b(1-\alpha'd)}.
\end{equation}
Since 
$$
m_{Y_{n,q}}(Y_{n,q})=m_{Y_{n}}(Y_{n})\cdot [\Gamma_n:\Gamma_n(q)],
$$
it follows from (\ref{eq:main2}) that
\begin{equation}\label{eq:T}
|B_n(x,\epsilon_n)\cap \bar\gamma\Gamma_n(q)|
=\frac{T_n(x)}{[\Gamma_n:\Gamma_n(q)]}
+O_{x,\eta}\left(m_{G_n^{\rm f}}(B_n^{\rm f})^{1-\alpha'(\kappa+d)+\eta}\right).
\end{equation}

Every $w\in \Z[1/n]$ can be uniquely written as $w=u\cdot [w]_n$ where $u$ is a unit in $\Z[1/n]$
and $[w]_n\in \mathbb{Z}_{\ge 0}$ is coprime to $n$.
Let $f=f_1\cdots f_t$ and 
$\mathcal{P}_{n,z}$ be the set of prime numbers which are coprime to $\Delta_n(f)n$ and bounded by $z$.
We denote by $S_{n,z}(x)$ the cardinality of the set of $\gamma\in B_n(x,\epsilon_n)\cap \Gamma_n$
such that $[f(\gamma)]_n$ is coprime to $\mathcal{P}_{n,z}$
(equivalently, $f(\gamma)$ is coprime to $\mathcal{P}_{n,z}$ in the ring $\Z[1/n]$).

We will apply the combinatorial sieve as in \cite[Th.~7.4]{HalRic} (see also \cite[Sec.~2]{NevSar}) to estimate
the quantity $S_{n,z}(x)$. For this we let 
$$
a_k:=| \{\gamma \in B_n(x,\epsilon_n)\cap\Gamma_n :\, [f(\gamma)]_n=k\}|.
$$
Then $T_n(x)=\sum_{k \ge 0} a_k$.
To apply the combinatorial sieve we need to verify the following three conditions:

\begin{enumerate}
\item[$(A_0)$] For every square-free $q$ divisible only by primes in $\mathcal{P}_{n,z}$,
\begin{equation}
\sum_{k=0 \bmod q} a_k=\frac{\rho(q)}{q}T_n(x)+R_q,
\end{equation}
where $\rho(q)$ is a nonnegative multiplicative function such that for primes $p\in \mathcal P_{n,z}$,
there exists $c_1<1$ satisfying
\begin{equation}\label{eqn:c1}
\frac{\rho(p)}{p} \le c_1.
\end{equation}

\item[$(A_1)$] Summing over square-free $q$ divisible only by primes in $\mathcal P_{n,z}$,
$$\sideset{}{'}\sum_{q \le T_n(x)^\tau} |R_q| \le c_2 T_n(x)^{1-\zeta}$$
for some $c_2,\tau,\zeta>0$.

\item[$(A_2)$] For some $w \in [2,z]$,
\begin{equation}\label{eqn:ellle}
-l \le \sum_{p \in \mathcal P_{n,z}:w\le p<z} \frac{\rho(p) \log p}{p}-t\log \frac{z}{w} \le c_3
\end{equation}
for some $c_3,l,t>0$.
\end{enumerate}

Once the  conditions $(A_0),(A_1),$ and $(A_2)$ are verified, by \cite[Th.~7.4]{HalRic},
for $z=T_n(x)^{\tau/s}$ with $s>9t$, we have the following estimate
\begin{equation}\label{eq:S}
S_{n,z}(x)\ge T_n(x)W(z)\left(C_1-C_2l\frac{(\log\log 3T_n(x))^{3t+2}}{\log T_n(x)}\right),
\end{equation}
where
$$
W(z)=\prod_{p\in \mathcal P_{n,z} : p \le z}\left(1-\frac{\rho(p)}{p}\right),
$$
and the constants $C_1,C_2>0$ are determined by $c_1,c_2,c_3,\tau,\zeta,t.$


We denote by $\Gamma_n^{(q)}$ the image of $\Gamma_n$ in $\hbox{GL}_N(\Z[1/n]/q\Z[1/n])$.
Note that 
$$
\Gamma_n^{(q)}\simeq \Gamma_n/\Gamma_n(q).
$$
To verify $(A_0)$, we observe that (\ref{eq:T}) implies that
\begin{align*}
\sum_{k=0\bmod q}a_k&=|\{\gamma \in B_n(x,\epsilon_n)\cap \Gamma_n :\, f(\gamma)=0\bmod q\}|\\
&=\sum_{\bar\gamma \in \Gamma_n^{(q)}: f(\bar\gamma)=0 \bmod q} | B_n(x,\epsilon_n)\cap \bar\gamma\Gamma_n(q)|\\
&=|\Gamma_n^{(q)}\cap \{f=0\}|\left(\frac{T_n(x)}{[\Gamma_n:\Gamma_n(q)]}+O_{x,\eta}\left(m_{G_n^{\rm
        f}}(B_n^{\rm f})^{1-\alpha'(\kappa+d)+\eta}\right)\right)\\
&=\frac{\rho(q)}{q}T_n(x)+O_{x,\eta} \left(|\Gamma_n^{(q)}\cap \{f=0\}|\cdot 
m_{G_n^{\rm f}}(B_n^{\rm f})^{1-\alpha'(\kappa+d)+\eta}\right),
\end{align*}
where 
$$
\rho(q):=\frac{q|\Gamma_n^{(q)}\cap \{f=0\}|}{[\Gamma_n:\Gamma_n(q)]}.
$$
Since $q$ is coprime to $\Delta_n(f)$, the polynomial $f$ 
is not identically zero on $\Gamma_n^{(q)}$ and $\rho(q)<q$.
Moreover, it follows from the strong approximation property that 
\begin{equation}\label{eq;ssplit}
\Gamma_n^{(q)}\simeq \prod_{p|q} \Gamma_n^{(p)},
\end{equation}
and for all but finitely many primes $p$,
$$
\Gamma_n^{(p)}\simeq{\rm G}^{(p)}(\F_p),
$$
where ${\rm G}^{(p)}$ denotes the reduction of ${\rm G}$ modulo $p$.
It follows from (\ref{eq;ssplit}) that the function $\rho$ is multiplicative

We observe that 
$$
{\rm G}^{(p)} \cap \{f=0\}= {\rm X}^{(p)}_1 \cup\cdots \cup {\rm X}^{(p)}_t
$$
where the varieties ${\rm X}^{(p)}_i={\rm G}^{(p)}\cap \{f_i=0\}$ are absolutely irreducible for all
but finitely many $p$ by the Noether theorem. Therefore, using the Lang--Weil estimate \cite{lw}, we deduce that
$$
|X^{(p)}_i(\mathbb{F}_p)|=p^{\dim({\rm G})-1}+O(p^{\dim({\rm G})-3/2}).
$$
Moreover, since $f_i$'s are distinct, for $i\ne j$, $\dim (X^{(p)}_i\cap X^{(p)}_j)\le \dim({\rm G})-2$
and 
$$
|(X^{(p)}_i\cap X^{(p)}_j)(\mathbb{F}_p)|=O(p^{\dim({\rm G})-2}).
$$
Hence,
$$
|{\rm G}^{(p)}(\F_p)\cap \{f=0\}|=t p^{\dim({\rm G})-1}+O(p^{\dim({\rm G})-3/2}).
$$
Since 
$$
|{\rm G}^{(p)}(\F_p)|=p^{\dim({\rm G})}+O(p^{\dim({\rm G})-1/2}),
$$
we deduce that 
\begin{equation}\label{eqn:rhop}
\rho(p)=t+O(p^{-1/2}).
\end{equation}
Since $\rho(p)<p$ for all primes $p\in \mathcal{P}_{n,z}$, this also implies that 
\eqref{eqn:c1} holds.
This establishes $(A_0)$ with 
$$
R_q=O_{x,\eta} \left(|\Gamma_n^{(q)}\cap \{f=0\}|\cdot 
m_{G_n^{\rm f}}(B_n^{\rm f})^{1-\alpha'(\kappa+d)+\eta}\right).
$$
It follows from (\ref{eq:TT}) that
\begin{align*}
  \sideset{}{'}\sum_{q\le T_n(x)^\tau} |R_q| &\ll_{x,\eta} \sum_{q\le T_n(x)^\tau} q^{d}
m_{G_n^{\rm f}}(B_n^{\rm f})^{1-\alpha'(\kappa+d)+\eta}\\
&\ll_{x,\eta} (T_n(x)^\tau)^{d+1} T_n(x)^{(1-\alpha'(\kappa+d)+\eta)/(1-\alpha' d)}.
\end{align*}
Since $\eta$ can be taken to be arbitrary positive number, we conclude
that
$$
 \sideset{}{'}\sum_{q\le T_n(x)^\tau} |R_q| \ll_{x,\eta} T_n(x)^{1-\zeta}
$$
with some $\zeta>0$, when
\begin{equation}\label{eqn:tau}
\tau<\tau_0:=(1-(1-\alpha'(\kappa+d))/(1-\alpha' d))/(d+1).
\end{equation}
Note that since $\kappa>0$, it follows that $\tau_0>0$.
This proves $(A_1)$. From (\ref{eq:kappa}),
$$
\tau_0=\frac{(4\iota)^{-1}-\alpha' d}{(d+1)^2(1-\alpha'd)}.
$$

It follows from (\ref{eqn:rhop}) and \cite[Th.~2.7(b)]{MonVau} that
$$
-c_3 \le \sum_{w\le p<z} \frac{\rho(p) \log p}{p}-t\log \frac{z}{w} \le c_3
$$
for some $c_3>0$. In particular, this implies the upper bound in (\ref{eqn:ellle}).
To establish the lower bound, we use the estimate
$$
\sum_{p|q} \frac{\log p}{p}=O(\log\log q)
$$
(see, for instance, \cite[Lem.~5.2]{GorNev}). This implies that condition $(A_2)$ holds with 
$$
l=O\left(\log\log (\Delta_n(f)n)\right)=O\left(\log\log n\right).
$$

Now we are in position to apply the main sieving argument (\ref{eq:S}).
Note that for $z=T_n(x)^{\tau/s}$, it follows from (\ref{eqn:rhop}) that
$$
W(z)\gg (\log z)^{-t}.
$$
Therefore, (\ref{eq:S}) gives
$$
S_{n,T_n(x)^{\tau/s}}(x)\ge \frac{T_n(x)}{(\log T_n(x))^{t}}\left(C_1-C_2'(\log\log n)\frac{(\log\log 3T_n(x))^{3t+2}}{\log T_n(x)}\right),
$$
Using (\ref{eq:final}), we deduce that for sufficiently large $n$,
$$
S_{n,T_n(x)^{\tau/s}}(x)\gg_x \frac{T_n(x)}{(\log T_n(x))^{t}}.
$$
Every $\gamma$ counted in $S_{n,T_n(x)^{\tau/s}}(x)$ satisfies 
\begin{equation}\label{eq:de}
d(x,\gamma)\le \epsilon_n\quad\hbox{and}\quad \gamma\in B_n^{\rm f}.
\end{equation}
This implies that $\hbox{den}(\gamma)=n$, and the numerator of $\gamma$ is $O_x(n)$.
In particular,
$$
[f(\gamma)]_n\ll_x n^{\deg(f)}.
$$
On the other hand, for any $\gamma$ counted in $S_{n,T_n(x)^{\tau/s}}(x)$,
all prime numbers $p$ which are coprime to $\Delta_n(f)$ and divide $[f(\gamma)]_n$ must satisfy
$$
p >z= T_n(x)^{\tau/s}.
$$
Thus, using (\ref{eq:final}), we deduce that the number of such prime factors is bounded from above by
\begin{align*}
\frac{\log (n^{\deg(f)})+O_x(1)}{\log (T_n(x)^{\tau/s})}&=\frac{s\deg(f)}{\tau b(1-\alpha'd)}+ o_x(1).
\end{align*}
Since this estimate holds for all $s>9t$, $b<a$ and $\tau<\tau_0$,
the number of such prime factors is at most
$$
\left\lceil\frac{9t\deg(f)}{\tau_0 a(1-\alpha'd)}\right\rceil
=\left\lceil \frac{9 t\deg(f)(d+1)^2}{a((4\iota)^{-1}-\alpha'd)}\right\rceil
$$
provided that $n$ is sufficiently large.
Therefore, we conclude that the element
$\gamma$ is $r$-prime in $\hbox{Mat}_N(\Z[1/n])$ with
$$
r=\delta_n(f)+\left\lceil \frac{9 t\deg(f)(d+1)^2}{a((4\iota)^{-1}-\alpha'd)} \right\rceil.
$$

For every $b\in (0,a)$ and sufficiently large $n$,
$$
\epsilon_n=m_{G_n^{\rm f}}(B_n^{\rm f})^{-\alpha'}\le n^{-b\alpha'}.
$$
Therefore, it follows from (\ref{eq:de}) that 
for every $\alpha<a\alpha_0'=\alpha_0$ and sufficiently large $n$,
$$
d(x,\gamma)\le n^{-\alpha}.
$$
This completes the proof of the theorem.
\end{proof}

When ${\rm G}$ is not assumed to be split over $\Q$, we have the following version of Theorem
\ref{thm:mainn}.

\begin{thm}\label{thm:mainn2}
Let $\rm G\subset {\rm GL}_N$ be a simply connected $\Q$-simple algebraic
group  defined over $\Q$, 
$f_1,\ldots,f_t$ a collection of polynomials as in the Introduction,
and $\mathcal{P}$ a finite collection of prime number such that ${\rm G}$
is isotropic over $\Q_p$ for all $p\in \mathcal{P}$.  
Then for every $x\in {\rm G}(\R)$, $\alpha\in (0,\alpha_0)$ with 
$\alpha_0:=d^{-1}a_{\mathcal{P}}(2\iota)^{-1}$ and admissible 
$n\ge n_0(\alpha,x)$ whose prime divisors are in $\mathcal{P}$, there exists $z\in {\rm G}(\Q)$ satisfying
\begin{align*}
&d(x,z)\le n^{-\alpha},\\
&\hbox{\rm den}(z)=n,
\end{align*}
and $r$-prime in ${\rm Mat}_N(\Z[1/n])$,
where
$$
r=\delta_n(f_1\cdots f_t)+\left\lceil \frac{9 t\deg(f_1\cdots f_t)(d+1)^2}{a_{\mathcal{P}}(2\iota)^{-1}-\alpha d} \right\rceil.
$$
The constant $n_0(\alpha,x)$ is uniform over $x$ in bounded  subsets of ${\rm G}(\R)$.
\end{thm}

The proof of Theorem \ref{thm:mainn2} goes along the same lines as the proof of Theorem \ref{thm:mainn},
but instead of the estimate on the averaging operators given by Theorem \ref{prop:QMET},
we use Theorem \ref{prop:QMET0}. This leads to slightly better estimates for the parameters $\alpha$ and $r$,
but the parameter $n_0(\alpha,x)$ now might depend on the set $\mathcal{P}$.

\end{document}